\documentclass[a4paper]{amsart}
\usepackage[T1]{fontenc}
\usepackage{amsmath,amssymb,amsthm,mathtools}
\usepackage{xcolor,tikz,caption}
\usepackage{hyperref}
\hypersetup{
colorlinks=true,
linkcolor=midori,
citecolor=hosyoku,
urlcolor=midori,
  pdftitle={The Fourier restriction problem for biradial functions},
  pdfauthor={Shinya Kinoshita, Juyoung Lee, Diogo Oliveira e Silva, Shobu Shiraki}}
\newtheorem{theorem}{Theorem}[section]
\newtheorem{lemma}[theorem]{Lemma}
\newtheorem{proposition}[theorem]{Proposition}
\newcommand{\R}{\mathbb R}
\newcommand{\Sp}{\mathbb S}
\newcommand{\dd}{\,\mathrm d}
\newcommand{\rMM}{\mathrm{MM}}
\newcommand{\rMR}{\mathrm{MR}}
\newcommand{\rRM}{\mathrm{RM}}
\newcommand{\rRR}{\mathrm{RR}}

\newcommand{\rR}{\mathrm{R}}

\newcommand{\cC}{\mathcal{C}}
\newcommand{\cE}{\mathcal{E}}

\newcommand{\cH}{\mathcal{H}}

\newcommand{\cT}{\mathcal{T}}

\newcommand{\cS}{\mathcal{S}}

\newcommand{\qstar}{q_\ast}

\usepackage{xcolor}
\usepackage{tikz}
\usetikzlibrary{positioning,intersections, calc, arrows,decorations.markings}
\usepackage{tikz-3dplot}

\definecolor{midori}{RGB}{0, 155, 88}
\definecolor{hosyoku}{RGB}{155,0,67}

\title{The Fourier restriction problem for biradial functions}

\author[Kinoshita]{Shinya Kinoshita}
\address[Shinya Kinoshita]{Graduate School of Mathematics, Nagoya University, Furocho, Chikusa-ku, Nagoya, 464-8602 Aichi, Japan}
\email{kinoshita@math.nagoya-u.ac.jp}

\author[Lee]{Juyoung Lee}
\address[Juyoung Lee]{Department of Mathematics, University of British Columbia, Vancouver, BC, Canada}
\email{juyounglee@math.ubc.ca}

\author[Oliveira e Silva]{Diogo Oliveira e Silva}
\address[Diogo Oliveira e Silva]{Center for Mathematical Analysis, Geometry and Dynamical Systems \& Departamento de Matem\'atica, Instituto Superior T\'ecnico, Universidade de Lisboa, Av. Rovisco Pais 10, 1049-001 Lisbon, Portugal}
\email{diogo.oliveira.e.silva@tecnico.ulisboa.pt}

\author[Shiraki]{Shobu Shiraki}
\address[Shobu Shiraki]{Department of Mathematics, Faculty of Science, University of Zagreb, Bijeni\v{c}ka cesta 30, 10000 Zagreb, Croatia}
\email{shobu.shiraki@math.hr}

\date{}

\begin{document}
\begin{abstract}
We determine the exact strong-type Fourier restriction range for the unit sphere
$\Sp^{d-1}$ on the class of $O(d-k)\times O(k)$-invariant functions, for every
$d\geq4$ and $2\leq k\leq d-2$. This includes the critical estimates beyond the
biradial Stein--Tomas endpoint left open by Mandel and Oliveira e Silva.
The proof combines a weighted $L^2$ estimate for extension from a circular arc
with the full two-dimensional restriction theorem. An appendix gives a direct
 proof of the classically conjectured range for axially symmetric functions on the whole sphere.
\end{abstract}
\maketitle

\section{Introduction}
Symmetry can enlarge the range of Fourier restriction estimates beyond that
permitted for arbitrary functions. For biradial functions, independent rotations
in two complementary subspaces constrain angular concentration without eliminating
it. Mandel and the third author \cite{MO} identified the corresponding necessary
conditions and established the endpoint biradial Stein--Tomas estimate. Whether
those conditions suffice throughout the admissible range remained open.

Write $\R^d=\R^{d-k}\times\R^k$. A function $f$ is \emph{$k$-biradial} if
$f(y,z)=f_0(|y|,|z|)$, or equivalently if it is invariant under the action of the group
$G_k:=O(d-k)\times O(k)$. We consider the Fourier restriction estimate
\begin{equation}\label{e:restriction}
 \|\widehat f\|_{L^{p'}(\Sp^{d-1})}\leq C\|f\|_{L^{q'}(\R^d)}
\end{equation}
for biradial Schwartz functions $f$, where the sphere carries its usual surface
measure and primes denote H\"older conjugates.

\begin{theorem}\label{t:main}
Let $d\geq4$, $2\leq k\leq d-2$, and $m:=\min\{k,d-k\}$. For
$1\leq p,q\leq\infty$, estimate \eqref{e:restriction} holds for every
$k$-biradial Schwartz function if and only if one of the following conditions is satisfied:
\begin{align*}
\mathrm{(i)}\quad&\frac1q<\frac{m-1}{2m},
 &\frac1p+\frac1q&\leq1;\\[2pt]
\mathrm{(ii)}\quad&\frac1q=\frac{m-1}{2m},
 &\frac1p+\frac1q&<1;\\[2pt]
\mathrm{(iii)}\quad&\frac{m-1}{2m}<\frac1q<\frac{d-1}{2d},
 &\frac{d-m}{p}+\frac{d+m}{q}&\leq d-1.
\end{align*}
\end{theorem}

Set $q_m:=2(d+m)/(d+m-2)$. The new estimates cover the critical segment
\begin{equation}\label{e:new-segment}
 \frac1{q_m}<\frac1q<\frac{d-1}{2d},\qquad
 \frac{d-m}{p}+\frac{d+m}{q}=d-1.
\end{equation}
The necessary conditions are from \cite[Theorem~1.3]{MO}; cases \textup{(i)},
\textup{(ii)}, and the part of \textup{(iii)} with $q\geq q_m$ were established
in \cite[Corollary~1.1]{MO}. The rest follows from \eqref{e:new-segment} by
H\"older's inequality on the sphere in the dual formulation; see Figure~\ref{f:biradial}.

For general functions, the conjectured range is
\[
 \frac1q<\frac{d-1}{2d},\qquad
 \frac{d-1}{p}+\frac{d+1}{q}\leq d-1.
\]
Thus biradial symmetry enlarges the optimal range of exponents. This complements
Muscalu--Oliveira \cite[Theorem~1.5]{MO24}, who established the linear restriction
conjecture for the paraboloid under a full tensor-product hypothesis on the extension
data. 
The cases $k\in\{1,d-1\}$ retain the
classical necessary conditions. In Appendix~\ref{s:appendix}, we prove their
sufficiency for the larger class of axially symmetric functions, without a polar
support assumption.

Related optimal estimates were obtained by Mandel for weighted extension from the
circle \cite{Mandel25} and for the composition of restriction and extension on
biradial functions \cite{Mandel26}. We refer to \cite[Section~1.3]{MO} for earlier
work on symmetry and restriction.

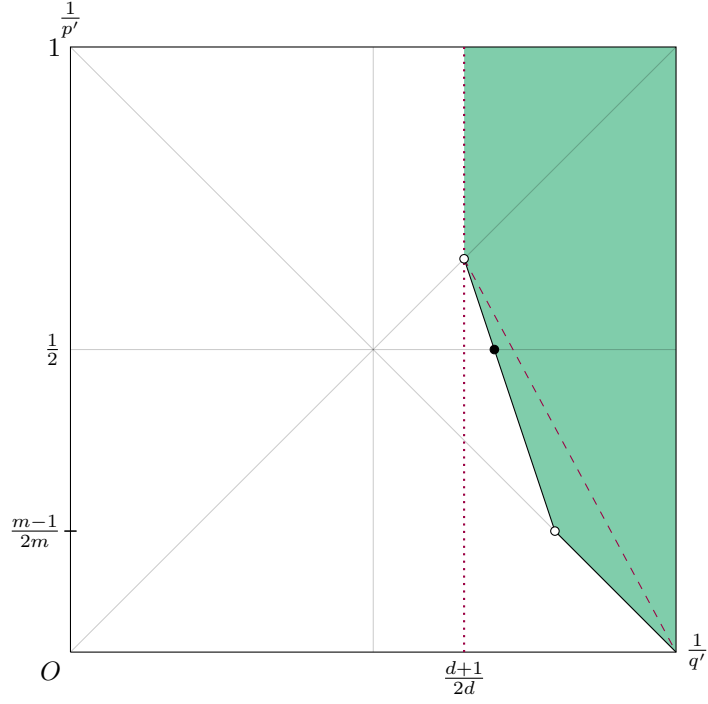
\begin{figure}[t]
\centering

\begin{tikzpicture}[scale=0.8em]
    \def\x{0.15}
    \def\y{0.8}
    \def\z{0.1}
    
    \coordinate (rest) at (1/2+\x,1/2+\x);
    \coordinate (new) at (\y,1-\y);

    \fill[opacity=0.5, color=midori] (1,0)--(1,1)--(1/2+\x,1)--(rest)--(new)--(1,0);

    \draw (0,0)--(0,1)--(1,1)--(1,0)--(0,0);
    \draw [opacity=0.2](1/2,0)--(1/2,1);
    \draw [opacity=0.2, name path=half] (0,1/2)--(1,1/2);
    \draw [opacity=0.2] (0,0)--(1,1);
    \draw [opacity=0.2] (1,0)--(0,1);

    \draw[dashed, hosyoku] (rest)--(1,0);
    \draw [name path=main] (rest)--(new)--(1,0);
    
    \draw [dotted,hosyoku,thick] (1/2+\x,0)--(1/2+\x,1); 

    \path[name intersections={of=half and main, by=st}];

    \fill [white] (rest) circle (0.2pt);
    \draw [black] (rest) circle (0.2pt);
    \fill [white] (new) circle (0.2pt);
    \draw [black] (new) circle (0.2pt);
    \fill [black] (st) circle (0.2pt);
    \draw [black] (st) circle (0.2pt);

    \node [below left] at (0,0) {$O$};
    \node [left] at (0,1/2) {$\frac12$};
    \node [left] at (0,1) {$1$};

    \node [right] at (1,0) {$\frac{1}{q'}$};
    \node [above] at (0,1) {$\frac{1}{p'}$};

    \node [below] at (1/2+\x,0) {$\frac{d+1}{2d}$};

    \draw (-0.01,1-\y)--(0.01,1-\y);
    \node [left] at ({0,1} |- new) {$\frac{m-1}{2m}$};

    \draw (-0.01,1-\y)--(0.01,1-\y);


\end{tikzpicture}
\captionof{figure}{
The shaded region illustrates the sharp range for the $k$-biradial restriction problem. Its critical boundary lies between the conjectured critical line for the classical restriction problem for general data (dashed) and the radial threshold (dotted). The black point denotes the biradial Stein--Tomas point, while the open circles indicate the excluded endpoints.
}\label{f:biradial}
\end{figure}

As in \cite{MO}, the proof reduces biradial spherical extension to a weighted
problem on a circular arc. The key is an angular input weight that compensates
for the degeneracy of the cosine phase in a weighted $L^2$ estimate. Interpolation
with the optimal two-dimensional restriction theorem gives critical estimates near
the diagonal endpoint, and interpolation with the biradial Stein--Tomas estimate
then fills \eqref{e:new-segment}. Section~\ref{s:reductions} sets up the reduction,
and Sections~\ref{s:MM}--\ref{s:RM} treat the four terms in the Bessel expansion.

\subsection*{Notation}
We write $\R_+=(0,\infty)$ and
$\widehat f(\xi)=\int_{\R^d}e^{-ix\cdot\xi}f(x)\dd x$.
The notation $A\lesssim B$ means $A\leq CB$, with $C$ depending only on $d,k$
and the fixed exponents; $A\sim B$ means both $A\lesssim B$ and $B\lesssim A$.
We use $L^{p,r}$ for Lorentz spaces.

\section{Reductions}\label{s:reductions}
Let
\[
\cE g(x):=\int_{\Sp^{d-1}}e^{ix\cdot\omega}g(\omega)\dd\sigma(\omega)
\]
be the extension operator. Averaging over $G_k$ is contractive on every Lebesgue
space under consideration and commutes with restriction and extension. Thus the
usual duality argument applies within the invariant subspaces. To prove the new
estimates in Theorem~\ref{t:main}, it suffices to establish
\begin{equation}\label{e:extension}
 \|\cE g\|_{L^q(\R^d)}\lesssim\|g\|_{L^p(\Sp^{d-1})}
\end{equation}
for $k$-biradial data and exponents given by \eqref{e:new-segment}.

\subsection{Two-dimensional weighted extension estimates}
Write
\[
 \omega=(\omega_y\cos\phi,\omega_z\sin\phi),\qquad
 \omega_y\in\Sp^{d-k-1},\quad \omega_z\in\Sp^{k-1},\quad 0<\phi<\pi/2.
\]
The surface measure is
\[
 \dd\sigma(\omega)=(\cos\phi)^{d-k-1}(\sin\phi)^{k-1}
 \dd\phi\dd\sigma(\omega_y)\dd\sigma(\omega_z).
\]
Indeed, the angular direction has unit length, whereas the tangent directions in
the two spherical factors are scaled by $\cos\phi$ and $\sin\phi$.
For biradial $g$, define its profile by
$g(\omega_y\cos\phi,\omega_z\sin\phi)=g_0(\phi)$. Then, for finite $p$,
\[
 \|g\|_{L^p(\Sp^{d-1})}^p\sim
 \int_0^{\frac\pi2}|g_0(\phi)|^p(\cos\phi)^{d-k-1}(\sin\phi)^{k-1}\dd\phi.
\]
For $x=(x_1,x_2)\in\R^{d-k}\times\R^k$, put $r_1=|x_1|$ and $r_2=|x_2|$.
The Fourier--Bessel identity
\begin{equation}\label{e:FourierBessel}
 \int_{\Sp^{n-1}}e^{ir\xi\cdot\omega}\dd\sigma(\omega)
 =
 (2\pi)^{\frac n2}
 (r|\xi|)^{1-\frac n2}J_{\frac n2-1}(r|\xi|)
\end{equation}
applied in both spherical variables gives
\begin{align*}
 \cE g(x_1,x_2)
 &=C_{d,k}r_1^{1-\frac{d-k}{2}}r_2^{1-\frac k2}\\
 &\quad\times\int_0^{\frac\pi2}g_0(\phi)
 J_{\frac{d-k}{2}-1}(r_1\cos\phi)J_{\frac k2-1}(r_2\sin\phi)\\[-2pt]
 &\hspace{20mm}\times(\cos\phi)^{\frac{d-k}{2}}(\sin\phi)^{\frac k2}\dd\phi.
\end{align*}
In particular, $Eg$ is biradial. In the rest of the proof, we identify it with its
profile in $(r_1,r_2)$ when taking radial integrals.

For $p,q>2$, introduce the auxiliary exponents $P,Q$ by
\begin{equation}\label{e:half}
 \frac1P:=\frac12-\frac1p,\qquad \frac1Q:=\frac12-\frac1q.
\end{equation}
The critical relation
\begin{equation}\label{e:critical}
 \frac{d-m}{p}+\frac{d+m}{q}=d-1
\end{equation}
then becomes
\begin{equation}\label{e:critical-aux}
 \frac{d-m}{P}+\frac{d+m}{Q}=1.
\end{equation}
Set
\[
 h(\phi):=g_0(\phi)(\cos\phi)^{\frac{d-k-1}{2}}(\sin\phi)^{\frac{k-1}{2}}
\]
and
\begin{align*}
 \cT h(r_1,r_2):=\int_0^{\frac \pi2}&h(\phi)J_{\frac{d-k}{2}-1}(r_1\cos\phi)
 J_{\frac k2-1}(r_2\sin\phi)\\[-2pt]
 &\hspace{20mm}\times(r_1\cos\phi)^{\frac12}(r_2\sin\phi)^{\frac12}\dd\phi.
\end{align*}
Thus
\[
 \cE g(x_1,x_2)=C_{d,k}r_1^{-\frac{d-k-1}{2}}r_2^{-\frac{k-1}{2}}\cT h(r_1,r_2),
\]
and polar coordinates in the two Euclidean blocks reveal that \eqref{e:extension}
is equivalent to
\begin{equation}\label{e:weighted}
 \big\|\cT h\,r_1^{-\frac{d-k-1}{Q}}r_2^{-\frac{k-1}{Q}}\big\|_{L^q(\R_+^2)}
 \lesssim
 \big\|h(\cos\phi)^{-\frac{d-k-1}{P}}(\sin\phi)^{-\frac{k-1}{P}}\big\|_{L^p(0,\frac\pi2)}.
\end{equation}
Recall  $q_m:=2(d+m)/(d+m-2)$.
We shall prove the following formulation of the new estimates.
\begin{theorem}\label{t:restricted}
Let $(p,q)$ satisfy \eqref{e:critical} and
$\qstar:=2d/(d-1)<q<q_m$. Then \eqref{e:weighted} holds.
\end{theorem}
By interpolation with the biradial Stein--Tomas estimate from \cite{MO}, it is
enough to prove \eqref{e:weighted} at critical points arbitrarily close to
$(1/p,1/q)=((d-1)/(2d),(d-1)/(2d))$. In Sections~\ref{s:MM}--\ref{s:RM}, we
therefore work on \eqref{e:critical} with $q>\qstar$ sufficiently close to
$\qstar$. In particular, $2<p<q$ and $0<1/P<1/Q$.

\subsection{Localization and Bessel decomposition}
Split $h$ into its restrictions to $(0,\pi/4)$ and $(\pi/4,\pi/2)$.
The second interval reduces to the first under
\begin{equation}\label{e:substitution}
 \phi\mapsto\frac \pi2-\phi,\qquad k\leftrightarrow d-k,
 \qquad r_1\leftrightarrow r_2.
\end{equation}
For $h$ supported in $(0,\pi/4)$, we have $\cos\phi\sim1$ and $\sin\phi\sim\phi$, and
the desired bound becomes
\[
 \big\|\cT h\,r_1^{-\frac{d-k-1}{Q}}r_2^{-\frac{k-1}{Q}}\big\|_{L^q(\R_+^2)}
 \lesssim\big\|h\phi^{-\frac{k-1}{P}}\big\|_{L^p(0,\frac\pi4)}.
\]
For every fixed $\nu>-1/2$, the standard Bessel asymptotics give
\begin{equation}\label{e:Bessel}
 t^{\frac12}J_\nu(t)=M_\nu(t)+R_\nu(t),\qquad
 M_\nu(t)=c_{\nu,+}e^{it}+c_{\nu,-}e^{-it},
\end{equation}
where the constants depend on $\nu$, and
\begin{equation}\label{e:remainder}
 |R_\nu(t)|\lesssim (1+t)^{-1},\qquad t>0;
\end{equation}
see \cite[Appendices~B.6 and~B.8]{Grafakos}. Expanding the two Bessel factors gives four types of contributions. We denote
representative terms by $\cT_\rMM,\cT_\rMR,\cT_\rRM,\cT_\rRR$. Each main factor is a
sum of two exponentials, and below we write out the positive-sign terms. The
full contributions are finite linear combinations of these sign terms. The same
proofs apply to all signs, since the kernel bounds are unchanged and the
unweighted circle estimates hold on the whole plane.

Unlike the decomposition in \cite{MO}, \eqref{e:Bessel} does not retain a
small-argument cutoff. The resulting separate terms need not preserve the
cancellation of the full Bessel product near $r_1=0$ or $r_2=0$. Accordingly,
our elementary remainder bounds require the local integrability conditions in
\eqref{e:radial-range} below, which hold near the diagonal endpoint. This is
sufficient for the interpolation argument.

\subsection{Two useful lemmas}
We use the following bilinear form of the Stein--Weiss inequality in $\R$; see
\cite[Theorem~B$_2^*$]{SteinWeiss}.
\begin{lemma}[Weighted Hardy--Littlewood--Sobolev inequality]\label{l:wHLS}
Let $1<r,s<\infty$ and $0<\lambda<1$. Suppose
\[
 \alpha<\frac1{r'},\quad \beta<\frac1{s'},\quad
 \alpha+\beta\geq0,\quad \frac1r+\frac1s\geq1,
 \qquad \frac1r+\frac1s+\alpha+\beta+\lambda=2.
\]
Then
\[
 \iint_{\R^2}\frac{|f(x)g(y)|}{|x|^\alpha|y|^\beta|x-y|^\lambda}
 \dd x\dd y\lesssim\|f\|_{L^r(\R)}\|g\|_{L^s(\R)}.
\]
\end{lemma}

We also recall weighted interpolation for linear operators; see \cite[Theorem 2]{Stein56}.
\begin{lemma}\label{l:interpolation}
Let $1\leq p_0,p_1,q_0,q_1\leq\infty$, and let $v_0,v_1,w_0,w_1$ be positive,
finite measurable weights on the respective measure spaces. Suppose a linear
operator $T$ is defined on the sum of the two input spaces and satisfies
\[
 \|Tf\,w_j\|_{L^{q_j}}\lesssim\|fv_j\|_{L^{p_j}},\qquad j\in\{0,1\}.
\]
For $0<\theta<1$, put
\[
 \frac1p=\frac{1-\theta}{p_0}+\frac\theta{p_1},\qquad
 \frac1q=\frac{1-\theta}{q_0}+\frac\theta{q_1},\qquad
 v=v_0^{1-\theta}v_1^\theta,\quad w=w_0^{1-\theta}w_1^\theta.
\]
Then $\|Tf\,w\|_{L^q}\lesssim\|fv\|_{L^p}$.
\end{lemma}

All oscillatory kernel calculations below can first be made with exponential
damping. For $0<c<1$, $\epsilon>0$, and $u\ne0$,
\[
 \int_0^\infty e^{-(\epsilon-iu)r}r^{-c}\dd r
 =\Gamma(1-c)(\epsilon-iu)^{c-1},\qquad
 \Big|\int_0^\infty e^{-(\epsilon-iu)r}r^{-c}\dd r\Big|
 \lesssim |u|^{c-1}.
\]
The bound is uniform in $\epsilon$, so the ensuing estimates and Fatou's lemma
justify removal of the damping. We initially take smooth inputs supported away
from the angular endpoints and then pass to the stated spaces by density.

\section{The main--main contribution}\label{s:MM}
Put
\[
 \Gamma:=\{(\cos\phi,\sin\phi):0\leq\phi\leq\frac\pi4\}.
\]
A representative main--main term is
\[
 \cT_\rMM h(r_1,r_2)=\cE_\Gamma h(r_1,r_2)
 :=\int_0^{\frac\pi4}h(\phi)e^{i(r_1\cos\phi+r_2\sin\phi)}\dd\phi.
\]
\begin{proposition}\label{p:MM}
Let $h$ be supported in $(0,\pi/4)$. For $(p,q)$ on \eqref{e:critical} with
$q>\qstar$ sufficiently close to $\qstar$,
\[
 \big\|\cT_\rMM h\,r_1^{-\frac{d-k-1}{Q}}r_2^{-\frac{k-1}{Q}}\big\|_{L^q(\R_+^2)}
 \lesssim \big\|h\phi^{-\frac{k-1}{P}}\big\|_{L^p(0,\frac\pi4)}.
\]
\end{proposition}

The main ingredient is a weighted $L^2$ estimate in which the input weight
compensates for the degeneracy of the cosine phase at $\phi=0$.
\begin{lemma}\label{l:weighted-L2}
Let $0<a_1,a_2<1/2$ satisfy $1/2<a_1+a_2<1$, and set $b=1-2a_1-a_2$. Then
\[
 \|\cE_\Gamma F\,r_1^{-a_1}r_2^{-a_2}\|_{L^2(\R_+^2)}
 \lesssim\|F\phi^{-b}\|_{L^2(0,\frac\pi4)}.
\]
\end{lemma}
\begin{proof}
Expanding the squared norm gives a quadratic form with kernel $K_1K_2$, where
\[
 K_1(\phi,\psi)=\int_0^\infty e^{ir(\cos\phi-\cos\psi)}r^{-2a_1}\dd r,
 \quad
 K_2(\phi,\psi)=\int_0^\infty e^{ir(\sin\phi-\sin\psi)}r^{-2a_2}\dd r.
\]
The damping convention above gives
\[
 |K_1K_2|\lesssim
 |\cos\phi-\cos\psi|^{2a_1-1}|\sin\phi-\sin\psi|^{2a_2-1}.
\]
For $\phi,\psi\in(0,\pi/4)$,
\[
 |\sin\phi-\sin\psi|\sim|\phi-\psi|,\qquad
 |\cos\phi-\cos\psi|\gtrsim(\phi\psi)^{\frac12}|\phi-\psi|.
\]
Consequently,
\[
 |K_1K_2|\lesssim
 (\phi\psi)^{-(\frac12-a_1)}|\phi-\psi|^{-2(1-a_1-a_2)}.
\]
Set $G(\phi)=F(\phi)\phi^{-b}$,
$\alpha=a_1+a_2-1/2$, and $\lambda=2(1-a_1-a_2)$. Then
$0<\alpha<1/2$, $0<\lambda<1$, and $2\alpha+\lambda=1$. Applying
Lemma~\ref{l:wHLS} with $r=s=2$ and $\beta=\alpha$, after extension by zero,
yields
\[
 \|\cE_\Gamma F\,r_1^{-a_1}r_2^{-a_2}\|_2^2
 \lesssim\int_0^{\frac\pi4}\!\int_0^{\frac\pi4}
 \frac{|G(\phi)G(\psi)|}{\phi^\alpha\psi^\alpha|\phi-\psi|^\lambda}
 \dd\phi\dd\psi\lesssim\|G\|_2^2.\qedhere
\]
\end{proof}

We also use the classical restriction theorem for the circle, in the formulation
of \cite[Theorem~1]{Mandel25}:
\begin{equation}\label{e:circle}
 \|\cE_\Gamma F\|_{L^t(\R^2)}\lesssim\|F\|_{L^s(0,\frac\pi4)},
 \qquad \frac1s+\frac3t=1,\quad 4<t\leq\infty.
\end{equation}
\begin{proof}[Proof of Proposition~\ref{p:MM}]
For sufficiently small $\epsilon>0$, set
\[
 \frac1s=\frac14+3\epsilon,\qquad \frac1t=\frac14-\epsilon,
\]
and choose $\theta\in(0,1)$ so that
\begin{equation}\label{e:critical-point}
 \frac1p=\frac{1-\theta}{s}+\frac\theta2,\qquad
 \frac1q=\frac{1-\theta}{t}+\frac\theta2
\end{equation}
lies on \eqref{e:critical}. Equivalently,
\begin{equation}\label{e:epsilon}
 \frac1P=(1-\theta)\left(\frac14-3\epsilon\right),\qquad
 \frac1Q=(1-\theta)\left(\frac14+\epsilon\right).
\end{equation}
Equation \eqref{e:critical-aux} requires
\[
 1=(1-\theta)\left(\frac d2-(2d-4m)\epsilon\right),
\]
which determines an admissible $\theta$ for small $\epsilon$. As $\epsilon\downarrow0$,
$\theta\to(d-2)/d$ and $P,Q\to2d$. These points approach the diagonal endpoint
from within \eqref{e:new-segment}.

Choose
\[
 a_1=\frac{d-k-1}{\theta Q},\qquad a_2=\frac{k-1}{\theta Q}.
\]
At $\epsilon=0$ their limiting values are
\[
 a_1=\frac{d-k-1}{2(d-2)},\qquad a_2=\frac{k-1}{2(d-2)}.
\]
Both lie strictly between $0$ and $1/2$, and their sum is $1/2$. The strict upper
bounds persist for small $\epsilon>0$. To check the remaining lower bound, use
$\theta=1-1/P-3/Q$ and \eqref{e:critical-aux} to obtain
\[
 \theta\bigl(2(a_1+a_2)-1\bigr)
 =(d-m-1)\left(\frac1Q-\frac1P\right)
 =4\epsilon(d-m-1)(1-\theta)>0.
\]
Lemma~\ref{l:weighted-L2} therefore applies with $b=1-2a_1-a_2$.
Interpolating it with \eqref{e:circle} by Lemma~\ref{l:interpolation} gives
\begin{equation}\label{e:MM-interpolated}
 \big\|\cE_\Gamma F\,r_1^{-\frac{d-k-1}{Q}}r_2^{-\frac{k-1}{Q}}\big\|_q
 \lesssim\|F\phi^{-\theta b}\|_p.
\end{equation}
Finally,
\[
 \theta b-\frac{k-1}{P}
 =1-\frac{k}{P}-\frac{2d-k}{Q}
 =(d-m-k)\left(\frac1P-\frac1Q\right)\leq0.
\]
Since $0<\phi<\pi/4<1$, the last input norm is bounded by the required one.
\end{proof}

\section{The remainder--remainder contribution}\label{s:RR}
Consider
\[
 \cT_\rRR h(r_1,r_2):=\int_0^{\frac\pi4}h(\phi)
 R_{\frac{d-k}{2}-1}(r_1\cos\phi)R_{\frac k2-1}(r_2\sin\phi)\dd\phi.
\]
\begin{proposition}\label{p:RR}
Let $h$ be supported in $(0,\pi/4)$. For $(p,q)$ on \eqref{e:critical} with
$q>\qstar$ sufficiently close to $\qstar$,
\begin{equation}\label{e:RR-goal}
 \big\|\cT_\rRR h\,r_1^{-\frac{d-k-1}{Q}}r_2^{-\frac{k-1}{Q}}\big\|_{L^q(\R_+^2)}
 \lesssim\big\|h\phi^{-\frac{k-1}{P}}\big\|_{L^p(0,\frac\pi4)}.
\end{equation}
\end{proposition}
\begin{proof}
By \eqref{e:remainder}, $\cos\phi\sim1$, and $\sin\phi\sim\phi$,
\[
 |\cT_\rRR h(r_1,r_2)|\lesssim
 \int_0^{\frac\pi4}\frac{|h(\phi)|}{(1+r_1)(1+r_2\phi)}\dd\phi.
\]
For $q>\qstar$ close enough to $\qstar$, we have
\begin{equation}\label{e:radial-range}
 \frac{d-1}{2d}>\frac1q>\frac{d-m-1}{2(d-m)}
 =\max\left\{\frac{k-1}{2k},\frac{d-k-1}{2(d-k)}\right\}.
\end{equation}
In particular,
\begin{equation}\label{e:radial-integral}
 \Big\|\frac{r_1^{-\frac{d-k-1}{Q}}}{1+r_1}\Big\|_{L^q(\R_+)}<\infty,
\end{equation}
because $(d-k-1)q/Q<1$ at zero, while $q>1$ suffices at infinity.
Similarly, scaling gives
\[
 \Big\|\frac{r_2^{-\frac{k-1}{Q}}}{1+r_2\phi}\Big\|_{L^q(\R_+)}
 \lesssim\phi^{\frac{k-1}{Q}-\frac1q}.
\]
Minkowski's inequality now bounds the left-hand side of \eqref{e:RR-goal} by
\[
 \int_0^{\frac\pi4}|h(\phi)|\phi^{\frac{k-1}{Q}-\frac1q}\dd\phi.
\]
H\"older's inequality completes the proof, since
\[
 \frac{k-1}{Q}-\frac1q+\frac{k-1}{P}+\frac1{p'}
 =k\left(\frac1P+\frac1Q\right)>0.
\]
\end{proof}

\section{The main--remainder contribution}\label{s:MR}
The next proposition is stated and proved in strong type. The proof keeps the
angular integral inside the remaining radial norm and then uses Young's
inequality in logarithmic coordinates.
\begin{proposition}\label{p:MR}
Let $h$ be supported in $(0,\pi/4)$. For $(p,q)$ on \eqref{e:critical} with
$q>\qstar$ sufficiently close to $\qstar$,
\[
 \big\|\cT_\rMR h\,r_1^{-\frac{d-k-1}{Q}}r_2^{-\frac{k-1}{Q}}\big\|_{L^q(\R_+^2)}
 \lesssim\big\|h\phi^{-\frac{k-1}{P}}\big\|_{L^p(0,\frac\pi4)}.
\]
\end{proposition}
Define the cosine operator
\begin{equation}\label{e:cosine}
 \cC F(r):=\int_0^{\frac\pi4}F(\phi)e^{ir\cos\phi}\dd\phi.
\end{equation}
\begin{lemma}\label{l:cosine}
Let $2\leq q<2(d-k)/(d-k-1)$, and set $1/Q=1/2-1/q$, allowing $Q=\infty$.
Then
\[
 \|\cC F\,r^{-\frac{d-k-1}{Q}}\|_{L^q(\R_+)}
 \lesssim\|F\phi^{-\eta}\|_{L^{q'}(0,\frac\pi4)},\qquad
 \eta=\frac1q-\frac{2(d-k-1)}Q.
\]
\end{lemma}
\begin{proof}
When $q=2$, the substitution $u=\cos\phi$ and Plancherel's theorem give
\[
 \|\cC F\|_{L^2(\R_+)}^2
 \lesssim\int_0^{\frac\pi4}\frac{|F(\phi)|^2}{\sin\phi}\dd\phi
 \sim\|F\phi^{-\frac12}\|_2^2.
\]
For $q>2$, first let $0<a<1/2$ and $b=1/2-2a$. We claim
\begin{equation}\label{e:cosine-L2}
 \|\cC F\,r^{-a}\|_{L^2(\R_+)}\lesssim\|F\phi^{-b}\|_{L^2(0,\frac\pi4)}.
\end{equation}
The squared norm has kernel
\[
 K(\phi,\psi)=\int_0^\infty e^{ir(\cos\phi-\cos\psi)}r^{-2a}\dd r,
 \qquad
 |K(\phi,\psi)|\lesssim(\phi\psi)^{-(\frac12-a)}|\phi-\psi|^{-(1-2a)}.
\]
With $G=F\phi^{-b}$, Lemma~\ref{l:wHLS} applies with $r=s=2$,
$\alpha=\beta=a$, and $\lambda=1-2a$, yielding \eqref{e:cosine-L2}.

Interpolate \eqref{e:cosine-L2} with
$\|\cC F\|_\infty\leq\|F\|_1$. Choose
\[
 \theta=\frac2q,\qquad
 a=\frac{d-k-1}{\theta Q}=\frac{(d-k-1)(q-2)}4.
\]
The stated range of $q$ is exactly what guarantees $0<a<1/2$.
Lemma~\ref{l:interpolation} gives the asserted bound, since
$\theta a=(d-k-1)/Q$ and $\theta b=\eta$.
\end{proof}

\begin{proof}[Proof of Proposition~\ref{p:MR}]
Put $L=\pi/4$ and write a representative term as
\[
 \cT_\rMR h(r_1,r_2)=\int_0^L
 h(\phi)e^{ir_1\cos\phi}R_{\frac k2-1}(r_2\sin\phi)\dd\phi.
\]
Set
\[
 a=\frac{d-k-1}{Q},\qquad c=\frac{k-1}{Q},\qquad \eta=\frac1q-2a,
\]
and put $u(\phi)=h(\phi)\phi^{-(k-1)/P}$ and $A=(k-1)/P-\eta$.
For $q$ sufficiently close to $\qstar$, \eqref{e:radial-range} holds. Applying
Lemma~\ref{l:cosine} for each $r=r_2$, followed by \eqref{e:remainder}, gives
\begin{align*}
 \|\cT_\rMR h\,r_1^{-a}r_2^{-c}\|_{L^q(\R_+^2)}\lesssim
 \Big\|r^{-c}\Big(\int_0^L
 \frac{|u(\phi)|^{q'}\phi^{Aq'}}{(1+r\phi)^{q'}}\dd\phi
 \Big)^{\frac{1}{q'}}\Big\|_{L^q(\R_+)}.
\end{align*}
Define
\[
 A_\ast:=\frac1p+\frac2q-1-c.
\]
The critical relation gives
\[
 A-A_\ast=(d-m-k)\left(\frac1Q-\frac1P\right)\geq0.
\]
Since $0<\phi<L<1$, replacing $A$ by $A_\ast$ only increases the last expression.
Extend $u$ by zero to $(0,\infty)$ and define
\[
 \cH u(r)
 :=
 r^{-c}\Big(\int_0^\infty
 \frac{|u(\phi)|^{q'}\phi^{A_\ast q'}}{(1+r\phi)^{q'}}\dd\phi
 \Big)^{\frac{1}{q'}}.
\]
Use logarithmic coordinates $\phi=e^t$, $r=e^{-s}$, and put
\[
 U(t)=e^{\frac tp}u(e^t),\qquad
 \delta=\frac1q-c,\qquad K(v)=\frac{e^{\delta q'v}}{(1+e^v)^{q'}}.
\]
By \eqref{e:radial-range}, $0<\delta<1$. Hence $K$ decays exponentially at
both ends of the real line. Since $A_\ast+1/q'-1/p=\delta$, direct substitution yields
\[
 \bigl[e^{-\frac sq}\cH u(e^{-s})\bigr]^{q'}
 =\int_\R |U(t)|^{q'}K(t-s)\dd t.
\]
Choose $\rho$ by $1/\rho=1+q'/q-q'/p$. Since $2<p<q$, we have
$1<p/q'<q/q'<\infty$ and $1<\rho<\infty$. Young's convolution inequality gives
\[
 \|\cH u\|_{L^q(\R_+)}^{q'}
 \lesssim\big\||U|^{q'}\big\|_{L^{p/q'}(\R)}\|K\|_{L^\rho(\R)}
 \lesssim\|U\|_{L^p(\R)}^{q'}=\|u\|_{L^p(\R_+)}^{q'}.
\]
This proves the required strong estimate.
\end{proof}

\section{The remainder--main contribution}\label{s:RM}
\begin{proposition}\label{p:RM}
Let $h$ be supported in $(0,\pi/4)$. For $(p,q)$ on \eqref{e:critical} with
$q>\qstar$ sufficiently close to $\qstar$,
\[
 \big\|\cT_\rRM h\,r_1^{-\frac{d-k-1}{Q}}r_2^{-\frac{k-1}{Q}}\big\|_{L^q(\R_+^2)}
 \lesssim\big\|h\phi^{-\frac{k-1}{P}}\big\|_{L^p(0,\frac\pi4)}.
\]
\end{proposition}
We use the sine operator
\[
 \cS F(r):=\int_0^{\frac\pi4}F(\phi)e^{ir\sin\phi}\dd\phi.
\]
\begin{lemma}\label{l:sine}
Let $2\leq q\leq\infty$ and $0<a<1/2$, and define $s$ by
$1/s=1+(2/q)(a-1/2)$. Then
\[
 \|\cS F\,r^{-\frac{2a}{q}}\|_{L^q(\R_+)}\lesssim\|F\|_{L^s(0,\frac\pi4)}.
\]
\end{lemma}
\begin{proof}
Set $r_a=2/(1+2a)$. Expanding the squared norm, using the regularized kernel bound
\[
 \Big|\int_0^\infty e^{ir(\sin\phi-\sin\psi)}r^{-2a}\dd r\Big|
 \lesssim|\phi-\psi|^{2a-1},
\]
and applying the one-dimensional weighted Hardy--Littlewood--Sobolev inequality gives
\begin{equation}\label{e:sine-L2}
 \|\cS F\,r^{-a}\|_{L^2(\R_+)}\lesssim\|F\|_{L^{r_a}(0,\frac\pi4)}.
\end{equation}
Here Lemma~\ref{l:wHLS} applies with $r=s=r_a$, $\alpha=\beta=0$,
and $\lambda=1-2a$. Interpolation with
$\|\cS F\|_\infty\leq\|F\|_1$, using parameter $2/q$, proves the result.
The endpoints $q=2$ and $q=\infty$ are the two estimates just stated.
\end{proof}

\begin{proof}[Proof of Proposition~\ref{p:RM}]
Take $a=(k-1)(q-2)/4$. For $q>\qstar$ close enough to $\qstar$,
$0<a<1/2$ and $2a/q=(k-1)/Q$. Applying Lemma~\ref{l:sine} for each fixed $r_1$
to
\[
 F(\phi)=h(\phi)R_{\frac{d-k}{2}-1}(r_1\cos\phi)r_1^{-\frac{d-k-1}{Q}}
\]
and using \eqref{e:remainder} and $\cos\phi\sim1$, we obtain
\begin{align*}
 &\|\cT_\rRM h\,r_1^{-\frac{d-k-1}{Q}}r_2^{-\frac{k-1}{Q}}\|_{L^q(\R_+^2)}\\
 &\qquad\lesssim
 \Big\|\frac{r_1^{-\frac{d-k-1}{Q}}}{1+r_1}\Big\|_{L^q(\R_+)}
 \|h\|_{L^s(0,\frac\pi4)},\qquad
 \frac1s=\frac{k+1}{2}-\frac{k}{q}.
\end{align*}
The radial norm is finite by \eqref{e:radial-range}--\eqref{e:radial-integral}.
Since $p,q>2$,
\[
 \frac1s-\frac1p=\frac{k+1}{2}-\frac{k}{q}-\frac1p>0.
\]
H\"older's inequality on the finite angular interval therefore gives
$\|h\|_s\lesssim\|h\|_p\lesssim\|h\phi^{-(k-1)/P}\|_p$.
\end{proof}

\begin{proof}[Completion of the proofs of Theorems~\ref{t:restricted} and~\ref{t:main}]
Sum the estimates in Propositions~\ref{p:MM}, \ref{p:RR}, \ref{p:MR}, and
\ref{p:RM} over the finitely many sign choices. The same estimates apply to the
second angular interval by \eqref{e:substitution}. This proves \eqref{e:weighted},
and hence \eqref{e:extension}, at critical points arbitrarily close to the diagonal
endpoint. All constants are uniform over the initial smooth approximations, so
density removes the auxiliary support restrictions.

For any fixed point on \eqref{e:new-segment}, choose one of the established points
closer to the diagonal endpoint and interpolate the original extension operator
between that point and the biradial Stein--Tomas estimate
$\cE:L^2(\Sp^{d-1})\to L^{q_m}(\R^d)$ from \cite{MO}. Interpolation is performed
on the invariant subspaces, or equivalently after composition with group averaging.
This proves Theorem~\ref{t:restricted}. H\"older's inequality on the sphere gives
the remaining new estimates in Theorem~\ref{t:main}. Its other cases and the
necessity of all the stated conditions follow from
\cite[Theorem~1.3 and Corollary~1.1]{MO}.
\end{proof}


\section*{Acknowledgements}

DOS and SS acknowledge the support provided during the Simons Semester
\emph{Contemporary Harmonic Analysis and its Synergies}, held in Poland from
August 1 to September 30, 2024; this event was part of the Simons Semesters at
the Banach Center, funded by the Simons Foundation (Award Number: 663281).
They thank the organizers for their kind hospitality.
SK was supported by JSPS KAKENHI grant JP24K16945.
DOS was funded by FCT/Portugal and the Recovery and Resilience Plan (PRR)
through projects UID/04459/2025 and UID/PRR/04459/2025, and by the project
2023.17881.ICDT (SHADE).
SS was supported by the Croatian Science Foundation under the project number HRZZ-IP-2022-10-5116 (FANAP).

The authors acknowledge the use of AI tools for proofreading. All mathematical ideas, arguments, and proofs in this work were developed, checked, and written by
the authors.


\appendix
\section{Axially symmetric functions}\label{s:appendix}
In this appendix we treat $k\in\{1,d-1\}$. More generally, write
$\R^d=\R^{d-1}\times\R$ and call $f$ \emph{axially symmetric} if
\[
 f(y,z)=f_0(|y|,z),\qquad (y,z)\in\R^{d-1}\times\R.
\]
Unlike $1$-biradial symmetry, axial symmetry does not require evenness in $z$.
A small spherical cap centered at a pole is compatible with axial symmetry, so
the usual Knapp example applies. Together with the constant extension datum,
it gives the classical necessary conditions
\[
 \frac1q<\frac{d-1}{2d},\qquad
 \frac{d-1}{p}+\frac{d+1}{q}\leq d-1.
\]
For $1$-biradial data, one may use the union of the two antipodal caps; see
\cite[Remark~6.1]{MO}. Stein--Tomas, interpolation with the trivial
$L^1(\Sp^{d-1})\to L^\infty(\R^d)$ extension bound, and H\"older's inequality
already give sufficiency when $q\geq2(d+1)/(d-1)$. The following theorem treats
the remaining critical segment.
\begin{theorem}\label{t:axial}
Let $d\geq3$ and suppose
\begin{equation}\label{e:axial-critical}
 \frac{d-1}{2(d+1)}<\frac1q<\frac{d-1}{2d},\qquad
 \frac{d-1}{p}+\frac{d+1}{q}=d-1.
\end{equation}
Then \eqref{e:restriction} holds for every axially symmetric Schwartz function.
\end{theorem}
Shao \cite[Corollary~2.3]{Shao09} established the restriction conjecture for
cylindrically symmetric data on the paraboloid. His corresponding spherical
estimates are localized near a pole, specifically to the lower third of the
sphere \cite[Remark~2.8 and Section~5]{Shao09}. We give a direct proof on the
whole sphere, based on Plancherel's identity and weighted interpolation.

\subsection{Reduction to a weighted estimate}
Parametrize the sphere by
\[
 \omega=(\vartheta\sin\varphi,\cos\varphi),\qquad
 0<\varphi<\pi,\quad \vartheta\in\Sp^{d-2},
\]
so that $\dd\sigma(\omega)=(\sin\varphi)^{d-2}\dd\varphi\dd\sigma(\vartheta)$.
For axially symmetric $g$, write $g(\omega)=g_0(\varphi)$ and set
\[
 \nu=\frac{d-3}{2},\qquad h(\varphi)=g_0(\varphi)(\sin\varphi)^{\frac{d-2}{2}}.
\]
Define
\[
 \cT h(r,z):=\int_0^\pi h(\varphi)(r\sin\varphi)^{\frac12}
 J_\nu(r\sin\varphi)e^{iz\cos\varphi}\dd\varphi,\qquad r>0,\quad z\in\R.
\]
By \eqref{e:FourierBessel},
$\cE g(y,z)=c_d|y|^{-(d-2)/2}\cT h(|y|,z)$. As in Section~\ref{s:reductions}, the
extension estimate is equivalent to
\begin{equation}\label{e:axial-weighted}
 \|\cT h\,r^{-\frac{d-2}{Q}}\|_{L^q(\R_+\times\R)}
 \lesssim\|h(\sin\varphi)^{-\frac{d-2}{P}}\|_{L^p(0,\pi)},
\end{equation}
where $P,Q$ are given by \eqref{e:half}. The critical relation becomes
\[
 \frac{d-1}{P}+\frac{d+1}{Q}=1.
\]

\subsection{Proof of Theorem~\ref{t:axial}}
The axial setting involves only one Bessel factor, and Plancherel's theorem
applies directly in the $z$-variable. We interpolate the next two estimates.
\begin{lemma}\label{l:axial-L2}
For $1/2<a<(d-1)/2$,
\begin{equation}\label{e:axial-L2}
 \|\cT h\,r^{-a}\|_{L^2(\R_+\times\R)}
 \lesssim\|h(\sin\varphi)^{a-1}\|_{L^2(0,\pi)}.
\end{equation}
\end{lemma}
\begin{proof}
For every $r>0$, Plancherel's identity in $z$ gives
\[
 \int_\R|\cT h(r,z)|^2\dd z
 =2\pi\int_0^\pi |h(\varphi)(r\sin\varphi)^{\frac12}
 J_\nu(r\sin\varphi)|^2\frac{\dd\varphi}{\sin\varphi}.
\]
Integrate against $r^{-2a}\dd r$ and substitute $\tau=r\sin\varphi$ to obtain
\[
 \|\cT h\,r^{-a}\|_2^2
 =
 2\pi\Big(\int_0^\infty|J_\nu(\tau)\tau^{\frac12-a}|^2\dd\tau\Big)
 \|h(\sin\varphi)^{a-1}\|_2^2.
\]
The Bessel integral is finite: near zero its integrand is
$O(\tau^{d-2-2a})$, and at infinity it is $O(\tau^{-2a})$.
The stated range of $a$ ensures both integrability conditions.
\end{proof}

\begin{lemma}\label{l:axial-unweighted}
Let $4<t<\infty$ and $1/s+3/t=1$. Then
\begin{equation}\label{e:axial-unweighted}
 \|\cT h\|_{L^t(\R_+\times\R)}\lesssim\|h\|_{L^s(0,\pi)}.
\end{equation}
\end{lemma}
\begin{proof}
Decompose the Bessel factor using \eqref{e:Bessel}. The two main terms are
constant multiples of
\[
 \cT_\pm h(r,z)=\int_0^\pi h(\varphi)e^{i(\pm r\sin\varphi+z\cos\varphi)}\dd\varphi.
\]
They are semicircle extension operators, so the classical planar restriction
theorem gives \eqref{e:axial-unweighted} for both. It remains to prove
\begin{equation}\label{e:axial-remainder}
 \|\cT_\rR h\|_{L^t(\R_+\times\R)}\lesssim\|h\|_{L^s(0,\pi)},
\end{equation}
where
\[
 \cT_\rR h(r,z):=\int_0^\pi h(\varphi)R_\nu(r\sin\varphi)e^{iz\cos\varphi}\dd\varphi.
\]
With $\tau=\cos\varphi$, this is
\[
 \cT_\rR h(r,z)=\int_{-1}^1H(\tau,r)e^{iz\tau}\dd\tau,\qquad
 H(\tau,r)=\frac{h(\arccos\tau)R_\nu(r\sqrt{1-\tau^2})}{\sqrt{1-\tau^2}}.
\]
Hausdorff--Young in $z$, Minkowski's inequality, and $t>t'$ imply
\begin{align*}
 \|\cT_\rR h\|_{L^t_{r,z}}
 \lesssim\big\|\|H(\cdot,r)\|_{L^{t'}_\tau}\big\|_{L^t_r}
 \leq\big\|\|H(\tau,\cdot)\|_{L^t_r}\big\|_{L^{t'}_\tau}
 \lesssim\|h(\varphi)(\sin\varphi)^{-\frac2t}\|_{L^{t'}(0,\pi)}.
\end{align*}
In the last step we used
$\|R_\nu(r\sin\varphi)\|_{L^t_r(\R_+)}\lesssim(\sin\varphi)^{-1/t}$,
which follows from \eqref{e:remainder}. Now
\[
 \frac1{t'}-\frac1s=\frac2t,\qquad
 (\sin\varphi)^{-\frac2t}
 \in 
 L^{\frac t2,\infty}(0,\pi).
\]
Lorentz--H\"older therefore gives $\cT_\rR:L^{s,1}(0,\pi)\to L^t(\R_+\times\R)$;
see \cite[Section~1.4]{Grafakos}.

For the strong estimate, choose $4<t_0<t<t_1<\infty$, with
$1/s_j+3/t_j=1$. Real interpolation between the corresponding bounds
$\cT_\rR:L^{s_j,1}\to L^{t_j}$ gives
$\cT_\rR:L^{s,s}\to L^{t,s}$; see \cite[Section~1.4.4]{Grafakos}.
Since $L^{s,s}=L^s$ and $s<t$, the embedding $L^{t,s}\subset L^t$ proves
\eqref{e:axial-remainder}.
\end{proof}

\begin{proof}[Completion of the proof of Theorem~\ref{t:axial}]
Put $u=1/P$ and $v=1/Q$. On \eqref{e:axial-critical},
\[
 (d-1)u+(d+1)v=1,\qquad 0<u<v.
\]
Choose
\[
 \theta=(d-2)(u+v),\qquad a=\frac{v}{u+v},\qquad
 t=2+4a,\qquad \frac1s=1-\frac3t.
\]
Then $0<\theta<1$, $1/2<a<1\leq(d-1)/2$, and $4<t<6$. Moreover,
\[
 \left(\frac1p,\frac1q\right)
 =(1-\theta)\left(\frac1s,\frac1t\right)+\theta\left(\frac12,\frac12\right),
 \qquad
 \theta a=\frac{d-2}{Q},\quad \theta(1-a)=\frac{d-2}{P}.
\]
Lemma~\ref{l:interpolation} applied to Lemmas~\ref{l:axial-L2} and
\ref{l:axial-unweighted} therefore yields \eqref{e:axial-weighted} throughout the
critical segment. Density and symmetry-restricted duality complete the proof.
Together with the estimates preceding Theorem~\ref{t:axial} and H\"older's
inequality on the sphere, this gives the full classical strong-type range in the
axially symmetric class.
\end{proof}

\end{document}